\documentclass[12pt,a4paper]{amsart}%
\usepackage{eurosym}
\usepackage{graphicx,amscd,color,amsmath,amsfonts,amssymb,geometry,hyperref,soul}
\usepackage{amsmath}
\usepackage{amsfonts}
\usepackage{amssymb}
\usepackage{graphicx}%
\providecommand{\U}[1]{\protect\rule{.1in}{.1in}}
\newtheorem{theorem}{Theorem}
\theoremstyle{plain}

\newtheorem{corollary}{Corollary}

\newtheorem{remark}{Remark}

\numberwithin{equation}{section}
\begin{document}
\title[Some applications of the Regularity Principle in sequence spaces]{Some applications of the Regularity Principle in sequence spaces}
\author[Wasthenny Vasconcelos Cavalcante]{Wasthenny Vasconcelos Cavalcante}
\address[W.V. Cavalcante]{Department of Mathematics \\
\indent UFPE \\
\indent Recife, PE, Brazil.}
\thanks{Partially supported by Capes.}

\begin{abstract}
The Hardy--Littlewood inequalities for $m$-linear forms have their origin with
the seminal paper of Hardy and Littlewood (Q.J. Math, 1934). Nowadays it has
been extensively investigated and many authors are looking for the optimal
estimates of the constants involved. For $m<p\leq2m$ it asserts that there is
a constant $D_{m,p}^{\mathbb{K}}\geq1$ such that
\[
\left(  \sum_{j_{1},\cdots,j_{m}=1}^{n}\left\vert T(e_{j_{1}},\cdots,e_{j_{m}%
})\right\vert ^{\frac{p}{p-m}}\right)  ^{\frac{p-m}{p}}\leq D_{m,p}%
^{\mathbb{K}}\left\Vert T\right\Vert ,
\]
for all $m$--linear forms $T:\ell_{p}^{n}\times\cdots\times\ell_{p}%
^{n}\rightarrow\mathbb{K}=\mathbb{R}$ or $\mathbb{C}$ and all positive
integers $n$. Using a Regularity Principle recently proved by Pellegrino,
Santos, Serrano and Teixeira, we present a straightforward proof of the
Hardy--Littewood inequality and show that:

(1) if $m<p_{1}<p_{2}\leq2m$ then $D_{m,p_{1}}^{\mathbb{K}}\leq D_{m,p_{2}%
}^{\mathbb{K}}$;

(2) $D_{m,p}^{\mathbb{K}}\leq D_{m-1,p}^{\mathbb{K}}$ whenever $m<p\leq
2\left(  m-1\right)  $ for all $m\geq3$.

\end{abstract}
\maketitle


\section{Introduction}

Littlewood's $4/3$ inequality \cite{LLL} is probably the forerunner of the by
now so-called Hardy--Littlewood inequalities for multilinear forms. Published
in 1930, it asserts that%
\[
\left(  \sum_{i,j=1}^{n}\left\vert T(e_{i},e_{j})\right\vert ^{\frac{4}{3}%
}\right)  ^{\frac{3}{4}}\leq\sqrt{2}\sup_{\left\Vert x\right\Vert ,\left\Vert
y\right\Vert \leq1}\left\vert T\left(  x,y\right)  \right\vert
\]
for all bilinear forms $T:\ell_{\infty}^{n}\times\ell_{\infty}^{n}%
\rightarrow\mathbb{C}$ for all positive integers $n$. One year later,
Bohnenblust and Hille \cite{bh} generalized Littlewood's inequality to
$m$-linear forms and, in 1934, Hardy and Littlewood \cite{hl} extended
Littlewood's $4/3$ inequality to bilinear forms acting on $\ell_{p}^{n}$
spaces. In 1981, Praciano-Pereira \cite{praciano} extended Hardy--Littlewood's
inequalities to $m$-linear forms and, in 2016 Dimant and Sevilla-Peris
\cite{dimant} completed the results of Praciano-Pereira.

These results can be summarized as follows: for any integer $m\geq2$ there
exist (optimal) constants $C_{m,p}^{\mathbb{K}},D_{m,p}^{\mathbb{K}}\geq1$
such that
\begin{equation}
\left(  \sum_{j_{1},\cdots,j_{m}=1}^{n}\left\vert T(e_{j_{1}},\cdots,e_{j_{m}%
})\right\vert ^{\frac{2mp}{mp+p-2m}}\right)  ^{\frac{mp+p-2m}{2mp}}\leq
C_{m,p}^{\mathbb{K}}\left\Vert T\right\Vert , \label{i99}%
\end{equation}
when $2m\leq p\leq\infty$, and%
\begin{equation}
\left(  \sum_{j_{1},\cdots,j_{m}=1}^{n}\left\vert T(e_{j_{1}},\cdots,e_{j_{m}%
})\right\vert ^{\frac{p}{p-m}}\right)  ^{\frac{p-m}{p}}\leq D_{m,p}%
^{\mathbb{K}}\left\Vert T\right\Vert , \label{599}%
\end{equation}
when $m<p\leq2m$, for all $m$--linear forms $T:\ell_{p}^{n}\times\cdots
\times\ell_{p}^{n}\rightarrow\mathbb{K}$ and all positive integers $n$ (here,
and henceforth, $\mathbb{K}=\mathbb{R}$ or $\mathbb{C}$). The exponents are optimal.

The investigation of the optimal constants of the Hardy--Littlewood
inequalities (see \cite{pos, araujo1, araujo2, araujo3, araujo4}) is motivated
by their connection with the important Bohnenblust-Hille inequality (see, for
instance \cite{caro, maia, communications, santos} and the references therein).

\bigskip The original estimates for $D_{m,p}^{\mathbb{K}}$ are
\begin{equation}
D_{m,p}^{\mathbb{K}}\leq\left(  \sqrt{2}\right)  ^{m-1}.\label{1111}%
\end{equation}
Recently, Albuquerque \textit{et al.\cite{alb} }(see also \cite{nunes}) have
proved that%
\begin{equation}
D_{m,p}^{\mathbb{K}}\leq2^{\frac{\left(  m-1\right)  \left(  p-m+1\right)
}{p}}.\label{1112}%
\end{equation}

\bigskip In this note we use a Regularity Principle recently proved in
\cite{regularity} to give a straightforward proof of (\ref{599}) and we also
show some new monotonicity properties of the constants $D_{m,p}^{\mathbb{K}}$.

\bigskip

\section{A straightforward proof of the Hardy--Littlewood inequality and new
mototonicity properties}

\bigskip From now on, if $p\in(1,\infty),$ the number $p^{\ast}$ is the
conjugate of $p$, i.e.,%
\[
\frac{1}{p}+\frac{1}{p^{\ast}}=1.
\]
We start off by giving a straightforward proof of (\ref{599}). It is folklore
that the case $p=2m$ in (\ref{i99}) can be re-written as a coincidence theorem
for multiple summing operators. It says that every continuous $m$-linear form
is multiple $\left(  2;\left(  2m\right)  ^{\ast}\right)  $-summing. By the
Regularity Principle from \cite{regularity} (more precisely, by the Inclusion
Theorem \cite[Proposition 3.4]{regularity}) we know that every multiple
$\left(  2;\left(  2m\right)  ^{\ast}\right)  $-summing form is multiple
$\left(  \frac{p}{p-m};p^{\ast}\right)  $-summing, when $m<p\leq2m$ with
standard domination of norms. This means that%
\[
\left(  \sum_{j_{1},\cdots,j_{m}=1}^{n}\left\vert T(e_{j_{1}},\cdots,e_{j_{m}%
})\right\vert ^{\frac{p}{p-m}}\right)  ^{\frac{p-m}{p}}\leq D_{m,2m}%
^{\mathbb{K}}\left\Vert T\right\Vert
\]
for all $m$--linear forms $T:\ell_{p}^{n}\times\cdots\times\ell_{p}%
^{n}\rightarrow\mathbb{K}$ and all positive integers $n.$ In other words, we
have just proved (\ref{599}). Moreover, we have shown that (\ref{599}) is a
consequence of (\ref{i99}).

Now we prove a monotonicity result:

\begin{theorem}
If $m<p_{1}\leq p_{2}\leq2m$ then $D_{m,p_{1}}^{\mathbb{K}}\leq D_{m,p_{2}%
}^{\mathbb{K}}$\bigskip.
\end{theorem}

\begin{proof}
If
\[
\left(  \sum_{j_{1},\cdots,j_{m}=1}^{n}\left\vert T(e_{j_{1}},\cdots,e_{j_{m}%
})\right\vert ^{\frac{p_{2}}{p_{2}-m}}\right)  ^{\frac{p_{2}-m}{p_{2}}}\leq
D_{m,p_{2}}^{\mathbb{K}}\left\Vert T\right\Vert
\]
for all $m$--linear forms $T:\ell_{p_{2}}^{n}\times\cdots\times\ell_{p_{2}%
}^{n}\rightarrow\mathbb{K}$ and all positive integers $n$, we know that every
continuous $m$-linear form is multiple $\left(  \frac{p_{2}}{p_{2}-m}%
;p_{2}^{\ast}\right)  $-summing with norm dominated by $D_{m,p_{2}%
}^{\mathbb{K}}.$ By \cite[Proposition 3.4]{regularity}, since%
\[
\frac{\left(  \frac{p_{2}}{p_{2}-m}\right)  \left(  \frac{p_{2}}{p_{2}%
-1}\right)  \left(  \frac{p_{1}}{p_{1}-1}\right)  }{\left(  \frac{p_{2}}%
{p_{2}-1}\right)  \left(  \frac{p_{1}}{p_{1}-1}\right)  +m\left(  \frac{p_{2}%
}{p_{2}-m}\right)  \left(  \frac{p_{2}}{p_{2}-1}\right)  -m\left(  \frac
{p_{2}}{p_{2}-m}\right)  \left(  \frac{p_{1}}{p_{1}-1}\right)  }=\frac{p_{1}%
}{p_{1}-m},
\]
and%
\[
p_{1}^{\ast}<\frac{m\left(  \frac{p_{2}}{p_{2}-m}\right)  \left(  \frac{p_{2}%
}{p_{2}-1}\right)  }{m\left(  \frac{p_{2}}{p_{2}-m}\right)  -\left(
\frac{p_{2}}{p_{2}-1}\right)  }=m^{\ast}\Leftrightarrow p_{1}>m
\]
we conclude that every continuous $m$-linear form is multiple $\left(
\frac{p_{1}}{p_{1}-m};p_{1}^{\ast}\right)  $-summing with norm dominated by
$D_{m,p_{2}}^{\mathbb{K}}.$ Thus%
\[
\left(  \sum_{j_{1},\cdots,j_{m}=1}^{n}\left\vert T(e_{j_{1}},\cdots,e_{j_{m}%
})\right\vert ^{\frac{p_{1}}{p_{1}-m}}\right)  ^{\frac{p_{1}-m}{p_{1}}}\leq
D_{m,p_{2}}^{\mathbb{K}}\left\Vert T\right\Vert
\]
for all $m$--linear forms $T:\ell_{p_{1}}^{n}\times\cdots\times\ell_{p_{1}%
}^{n}\rightarrow\mathbb{K}$ and all positive integers $n$. Thus%
\[
D_{m,p_{1}}^{\mathbb{K}}\leq D_{m,p_{2}}^{\mathbb{K}}.
\]

\end{proof}

\begin{remark}
The Inclusion Theorem \cite[Proposition 3.4]{regularity} was also proved
recently and independently, with a different technique,  by F. Bayart
\cite{bayart}.
\end{remark}

\section{Further monotonicity properties}

In this section we prove the following result:

\begin{theorem}
\label{944}Let $m\geq2$ be a positive integer and $m+1<p\leq2m.$ Then
\[
D_{m+1,p}^{\mathbb{K}}\leq D_{m,p}^{\mathbb{K}}.
\]

\end{theorem}

Denoting%
\begin{align*}
\lfloor x\rfloor &  :=\max\{n\in\mathbb{N}:n<x\}\\
\lceil x\rceil &  :=\min\{n\in\mathbb{N}:n\geq x\},
\end{align*}
we have the following corollary:

\begin{corollary}
Let $p\in(3,\infty).$ The finite sequence
\end{corollary}

\[
\left(  D_{m,p}^{\mathbb{K}}\right)  _{m=\lceil p/2\rceil}^{\lfloor
p-1\rfloor}%
\]
is decreasing.

\subsection{Proof of Theorem \ref{944}}

If $m<p\leq2m$ we know that%
\[
\left(  \sum_{j_{1},\cdots,j_{m}=1}^{n}\left\vert T(e_{j_{1}},\cdots,e_{j_{m}%
})\right\vert ^{\frac{p}{p-m}}\right)  ^{\frac{p-m}{p}}\leq D_{m,p}%
^{\mathbb{K}}\left\Vert T\right\Vert
\]
for all $m$--linear forms $T:\ell_{p}^{n}\times\cdots\times\ell_{p}%
^{n}\rightarrow\mathbb{K}$ and all positive integers $n$. If $E$ is any Banach
space, let us denote its dual by $E^{\ast}$ and the closed unit ball of
$E^{\ast}$ by $B_{E^{\ast}}.$ Since
\[
\frac{p}{p-m}\geq2
\]
we can prove that%
\[
\left(  \sum_{j_{1},\cdots,j_{m+1}=1}^{n}\left\vert T(e_{j_{1}},\cdots
,e_{j_{m}},x_{j_{m+1}})\right\vert ^{\frac{p}{p-m}}\right)  ^{\frac{p-m}{p}%
}\leq D_{m,p}^{\mathbb{K}}\left\Vert T\right\Vert \sup_{\varphi\in B_{E^{\ast
}}}\sum_{j_{m+1}=1}^{n}\left\vert \varphi(x_{j_{m+1}})\right\vert
\]
for all $m+1$--linear forms $T:\ell_{p}^{n}\times\cdots\times\ell_{p}%
^{n}\times E\rightarrow\mathbb{K}$ and all positive integers $n$. In fact, for
all positive integers $n,$ define the linear operator%
\[
u_{n}:E\rightarrow\ell_{\frac{p-m}{p}}%
\]
by%
\[
u_{n}(x)=\left(  T(e_{j_{1}},\cdots,e_{j_{m}},x\right)  _{j_{1},...,j_{m}%
=1}^{n}.
\]
Since $\frac{p-m}{p}\geq2$ we know that $\ell_{\frac{p-m}{p}}$ has cotype
$\frac{p-m}{p}.$ Moreover, it is folklore that (see, for instance, \cite[page
29]{david}) the cotype constant is $1:$
\[
C_{\frac{p-m}{p}}\left(  \ell_{\frac{p-m}{p}}\right)  =1.
\]
Therefore, by the definition of cotype, and letting, as usual, $r_{j}$ denote
the Rademacher functions, we have%
\begin{align*}
&  \left(  \sum\limits_{j=1}^{n}\left\Vert u_{n}(x_{j})\right\Vert ^{\frac
{p}{p-m}}\right)  ^{\frac{p-m}{p}}\\
&  \leq C_{\frac{p-m}{p}}\left(  \ell_{\frac{p-m}{p}}\right)  \left(
\int\nolimits_{[0,1]}\left\Vert \sum\limits_{j=1}^{n}r_{j}(t)u_{n}%
(x_{j})\right\Vert ^{2}dt\right)  ^{1/2}\\
&  \leq\max_{t\in\lbrack0,1]}\left\Vert \sum\limits_{j=1}^{n}r_{j}%
(t)u_{n}(x_{j})\right\Vert \\
&  \overset{\text{\cite[page 284]{lp}}}{\leq}\sup_{\omega\in B_{\ell
_{\frac{p-m}{p}}^{\ast}}}\sum_{j=1}^{n}\left\vert \omega(u_{n}(x_{j}%
))\right\vert \\
&  \leq\left\Vert u_{n}\right\Vert \sup_{\varphi\in B_{E^{\ast}}}\sum
_{j_{m+1}=1}^{n}\left\vert \varphi(x_{j_{m+1}})\right\vert \\
&  \leq\sup_{\left\Vert x\right\Vert \leq1}\left(  \sum_{j_{1},\cdots
,j_{m+1}=1}^{n}\left\vert T(e_{j_{1}},\cdots,e_{j_{m}},x)\right\vert
^{\frac{p}{p-m}}\right)  ^{\frac{p-m}{p}}\left(  \sup_{\varphi\in B_{E^{\ast}%
}}\sum_{j_{m+1}=1}^{n}\left\vert \varphi(x_{j_{m+1}})\right\vert \right)  \\
&  \leq\sup_{\left\Vert x\right\Vert \leq1}D_{m,p}^{\mathbb{K}}\left\Vert
T(\cdot,...,\cdot,x)\right\Vert \left(  \sup_{\varphi\in B_{E^{\ast}}}%
\sum_{j_{m+1}=1}^{n}\left\vert \varphi(x_{j_{m+1}})\right\vert \right)  \\
&  \leq D_{m,p}^{\mathbb{K}}\left\Vert T\right\Vert \sup_{\varphi\in
B_{E^{\ast}}}\sum_{j_{m+1}=1}^{n}\left\vert \varphi(x_{j_{m+1}})\right\vert .
\end{align*}
Since%
\[
\left(  \sum_{j_{1},\cdots,j_{m+1}=1}^{n}\left\vert T(e_{j_{1}},\cdots
,e_{j_{m}},x_{j_{m+1}})\right\vert ^{\frac{p}{p-m}}\right)  ^{\frac{p-m}{p}%
}=\left(  \sum\limits_{j=1}^{n}\left\Vert u_{n}(x_{j})\right\Vert ^{\frac
{p}{p-m}}\right)  ^{\frac{p-m}{p}}%
\]
we thus have%
\begin{equation}
\left(  \sum_{j_{1},\cdots,j_{m+1}=1}^{n}\left\vert T(e_{j_{1}},\cdots
,e_{j_{m}},x_{j_{m+1}})\right\vert ^{\frac{p}{p-m}}\right)  ^{\frac{p-m}{p}%
}\leq D_{m,p}^{\mathbb{K}}\left\Vert T\right\Vert \sup_{\varphi\in B_{E^{\ast
}}}\sum_{j_{m+1}=1}^{n}\left\vert \varphi(x_{j_{m+1}})\right\vert .\label{761}%
\end{equation}
By the Regularity Principle for sequence spaces \cite{regularity}, for
$p>m+1,$ we have%
\begin{align}
&  \left(  \sum_{j_{m+1}=1}^{n}\left(  \sum_{j_{1},\cdots,j_{m+1}=1}%
^{n}\left\vert T(e_{j_{1}},\cdots,e_{j_{m}},x_{j_{m+1}})\right\vert ^{\frac
{p}{p-m}}\right)  ^{\frac{p-m}{p}\times\frac{p}{p-\left(  m+1\right)  }%
}\right)  ^{\frac{p-\left(  m+1\right)  }{p}}\label{800}\\
&  \leq D_{m,p}^{\mathbb{K}}\left\Vert T\right\Vert \left(  \sup_{\varphi\in
B_{E^{\ast}}}\sum_{j_{m+1}=1}^{n}\left\vert \varphi(x_{j_{m+1}})\right\vert
^{\frac{p}{p-1}}\right)  ^{\frac{p-1}{p}}\nonumber
\end{align}
for all positive integers $n$. As a matter of fact, here we do not need the
Regularity Principle in its whole generality, and we provide a direct proof of
(\ref{800}) for the sake of completeness. Note that (\ref{761}) is equivalent
to%
\[
\left(  \sum_{j_{m+1}=1}^{n}\left\Vert u_{n}(x_{j_{m+1}})\right\Vert
^{\frac{p}{p-m}}\right)  ^{\frac{p-m}{p}}\leq D_{m,p}^{\mathbb{K}}\left\Vert
T\right\Vert \sup_{\varphi\in B_{E^{\ast}}}\sum_{j_{m+1}=1}^{n}\left\vert
\varphi(x_{j_{m+1}})\right\vert
\]
and by using a standard trick involving the H\"{o}lder inequality (see also
\cite[10.4 (Inclusion Theorem)]{Di}) we have (\ref{800}).

Choosing $E=\ell_{p}^{n}$, since%
\[
\sup_{\varphi\in B_{E^{\ast}}}\sum_{j_{m+1}=1}^{n}\left\vert \varphi
(e_{j_{m+1}})\right\vert ^{\frac{p}{p-1}}=1
\]
and since%
\[
\frac{p}{p-m}\leq\frac{p}{p-\left(  m+1\right)  }%
\]
we conclude that%
\[
\left(  \sum_{j_{1},\cdots,j_{m+1}=1}^{n}\left\vert T(e_{j_{1}},\cdots
,e_{j_{m}+1})\right\vert ^{\frac{p}{p-\left(  m+1\right)  }}\right)
^{\frac{p-\left(  m+1\right)  }{p}}\leq D_{m,p}^{\mathbb{K}}\left\Vert
T\right\Vert
\]
for all $m+1$--linear forms $T:\ell_{p}^{n}\times\cdots\times\ell_{p}%
^{n}\rightarrow\mathbb{K}$ and all positive integers $n$. Therefore%
\[
D_{m+1,p}^{\mathbb{K}}\leq D_{m,p}^{\mathbb{K}}.
\]

\end{document}